\documentclass[reqno, letter]{amsart}
\usepackage{amsmath, amssymb,amsthm, verbatim, comment, color
}

\usepackage{changes}
\usepackage{lipsum}% <- For dummy text
\usepackage{setspace}

\usepackage{enumerate}
\usepackage{color}

\usepackage{ hyperref}
\usepackage{caption}
\usepackage{accents}
\newtheorem{theorem}{Theorem}

\newtheorem{corollary}[theorem]{Corollary}
\newtheorem{definition}[theorem]{Definition}

\newtheorem{proposition}[theorem]{Proposition}

\theoremstyle{remark}
\newtheorem{remark}[theorem]{Remark}
\newtheorem{example}[theorem]{Example}

 \newcommand{\M}{\mathsf{M}}

 \renewcommand{\phi}{\varphi}

\newcommand{\E}{\mathbb{E}}
\renewcommand{\P}{\mathbb{P}}
\newcommand{\N}{\mathbb{N}}

\newcommand{\R}{\mathbb{R}}

\newcommand{\be}{\begin{equation}}
\newcommand{\ee}{\end{equation}}

\DeclareMathOperator{\supp}{supp}
\newcommand{\bea}{\begin{eqnarray}}
\newcommand{\bes}{\begin{subequations}}
\newcommand{\ees}{\end{subequations}}
\newcommand{\bgt}{\begin{gather}}
\newcommand{\egt}{\begin{gather}}
\newcommand{\eea}{\end{eqnarray}}
\newcommand{\beaa}{\begin{eqnarray*}}
\newcommand{\eeaa}{\end{eqnarray*}}

\usepackage{bbm}

\renewcommand{\epsilon}{\varepsilon}

\newcommand{\fourIdx}[5]{%
\setbox1=\hbox{\ensuremath{^{#1}}}%
 \setbox2=\hbox{\ensuremath{_{#2}}}%
 \setbox5=\hbox{\ensuremath{#5}}%
 \hspace{\ifnum\wd1>\wd2\wd1\else\wd2\fi}%
 \ensuremath{\copy5^{\hspace{-\wd1}\hspace{-\wd5}#1\hspace{\wd5}#3}%
 _{\hspace{-\wd2}\hspace{-\wd5}#2\hspace{\wd5}#4}%
 }}

\numberwithin{equation}{section}
\numberwithin{theorem}{section}

\renewcommand{\subset}{\subseteq}

\usepackage{subcaption}
\renewcommand{\M}{\rm{MT}}

\newcommand\blue{\color{black}}

\begin{document}

\title{Dual attainment for the martingale transport problem}
\thanks{The project has been generously supported by the European Research Council under the European Union's Seventh Framework Programme (FP7/2007-2013) / ERC grant agreement no. 335421. The first and the second author gratefully acknowledge support from the Austrian Science Foundation (FWF) through grant Y782 and the third author also acknowledge the support of St John's College Oxford.
}

\author{Mathias Beiglb\"ock}%\thanks{mathias.beiglboeck@tuwien.ac.at}
\author{Tongseok Lim}%\thanks{E-mail: tongseok.lim@maths.ox.ac.uk}
\author{Jan Ob\l \'oj}%\thanks{E-mail: jan.obloj@maths.ox.ac.uk}

%\thanks{...\copyright 2016 by the authors.}

\begin{abstract} We investigate existence of dual optimizers in one-dimensional martingale optimal transport problems. While \cite{BeNuTo16} established such existence for weak (quasi-sure) duality, \cite{BeHePe12} showed existence for the natural stronger (pointwise) duality may fail even in regular cases. We establish that (pointwise) dual maximizers exist when $y\mapsto c(x,y)$ is convex, or equivalent to a convex function. 
It follows that when marginals are compactly supported, the existence holds when the cost $c(x,y)$ is twice continuously differentiable in $y$. Further, this may not be improved as we give examples with $c(x,\cdot)\in C^{2-\epsilon}$, $\epsilon >0$, where dual attainment fails. Finally, when measures are compactly supported, we show that dual optimizers are Lipschitz if $c$ is Lipschitz. 
\end{abstract}
\maketitle

%\noindent\emph{MSC2010:} 60G42, 49N05

\section{Introduction} 

In recent years, there has been a significant interest in optimal transport problems where the transport plan is constrained to be a martingale. Referred to as martingale optimal transport (MOT), they were introduced by 
\cite{BeHePe12,GaHeTo13} to study the mathematical finance question of computing model--independent no--arbitrage price bounds, see \cite{Ho11} for a survey, and have been studied in many papers since, e.g.\ \cite{HoNe12,DoSo12,CaLaMa14, NuStTa17}. They are however of much wider mathematical interest. Mirroring classical optimal transport, they have important consequences for the study of martingale inequalities, see e.g.\ \cite{BeNu14,HeObSpTo12,ObSp14}. In continuous time, they are intimately linked with the Skorokhod embedding problem, see \cite{Ob04} for an overview of the latter, and have already led to new contributions to this well established field, see \cite{BeCoHu14}. 

Most papers on MOT either study the structure and geometry of optimisers or investigate a form of general Kantorovitch duality. Duality is of particular importance for mathematical finance: the primal problem corresponds to option pricing while the dual offers robust hedging strategies. However the latter poses a challenge: as already shown in \cite{BeHePe12}, the dual problem in MOT does not admit an optimiser in general. One way to recover the dual attainment is relaxing the duality and considering not pointwise but weaker, quasi--sure, inequalities, as shown in \cite{BeNuTo16}. 
Our aim here instead is to identify sufficient conditions on the problem under which a suitably nice dual optimiser exists. This has immediate applications in robust mathematical finance, where pointwise inequalities are more natural. Equally importantly, we believe, this problem is of intrinsic mathematical interest. In fact, answering such questions is an important prerequisite for the future development of the field and understanding geometry of primal optimisers, or existence of Brenier--type MOT plans in multiple dimensions. So far results in this direction are limited to dimension 1 and 2 \cite{HoNe12,HoKl13,BeJu16} and more recently \cite{GhKiLi16b}. However, the methods and results of \cite{GhKiLi16b} would allow to provide a satisfactory answer to this central question, conditionally on the existence of dual maximisers. 

To present in more detail the questions we want to study, we need to introduce some notation 
Let $\Omega:=\R\times\R$ be the canonical space and $(X,Y)$ the canonical process, i.e.\ $X(x,y)=x$ and $Y(x,y)=y$ for all $(x,y)\in\Omega$. We also denote by $P_\R$ and $P_\Omega$ the collections of all probability measures on $\R$ and $\Omega$, respectively. For fixed $\mu,\nu\in P_\R$ with finite first moments, we consider the following subsets of $P_\Omega$
 \begin{align}
 \Pi(\mu,\nu)
 &:=
 \big\{\P\in P_\Omega: X\sim_\P\mu, Y\sim_\P\nu\big\},
 \\
 \M(\mu,\nu)
 &:=
 \big\{\P\in\Pi(\mu,\nu): \E^\P[Y|X]=X,~\P-a.s. \big\}.
 \end{align}
The set $\Pi(\mu,\nu)$ is non-empty as it contains the product measure $\mu\otimes\nu$. By a classical result of Strassen \cite{St65}, $\M(\mu,\nu)$ is non-empty if and only if $\mu\preceq\nu$ in convex order:
 \begin{align}\label{convexorder}
 \mu(\xi)\le\nu(\xi)
 &\mbox{ for all convex function }
 \xi, \mbox{ where } \mu(\xi) := \int \xi(x)\mu(dx).
 \end{align}
Throughout we assume that $c:\Omega\rightarrow\R$ is a {\blue lower-semicontinuous} cost function with $c(x,y)\le a(x) + b(y)$ for some $a\in L^1(\mu)$ and $b\in L^1(\nu)$. Then $\E^\P[c(X,Y)]$ is a well-defined scalar in $\R\cup\{-\infty\}$.
The martingale optimal transport problem, as introduced in \cite{BeHePe12} in the present discrete-time case and in \cite{GaHeTo13} in continuous time, is defined by the following primal problem:
 \begin{align}\label{MGTP}
 \mathbf{P}
 \;:=\;
 \mathbf{P}(c)
 &:=
 \inf_{\P\in\M(\mu,\nu)} \E^\P[c(X,Y)].
 \end{align}
Its dual is given by
 \begin{equation}\label{dualproblem}
 \mathbf{D} 
 :=
 \mathbf{D}(c)
 := 
 \sup_{(f,g)\in D_c} \big\{\nu(g)-\mu(f)\big\},
 \end{equation}
where%, denoting $h^\otimes(x,y):=h(x)(y-x)$ for all $(x,y)\in\Omega$,
  \begin{align*}
 D_c:=
 \big\{&(f,g):~f^-\!\in L^1(\mu),~g^+\!\in L^1(\nu),
          ~\mbox{and for some } h \in L^\infty(\mu),\\
&g(y) - f(x) - h(x) \cdot (y-x)\le c(x,y)  \quad  \forall x\in \R, \forall y \in \R \big\}.
 \end{align*}
{\blue We assume that $P(c)$ is finite. Lower-semicontinuity\footnote{{\blue Throughout the paper lower-semicontinuity is only required to guarantee the existence of some  $\P^*\in\M(\mu,\nu)$ satisfying $\E^{\P^*}[c(X,Y)]=\mathbf{P}\in \R$. Hence we might work with a merely measurable cost function and assume existence of a primal minimizer.}} implies that the infimum in \eqref{MGTP} is attained i.e.\ $\mathbf{P}=\E^{\P^*}[c(X,Y)]$ for some $\P^*\in\M(\mu,\nu),$ cf.\ Remark \ref{PrimalAttained} below.}
 
  In mathematical finance, the cost $c$ has the interpretation of the payoff of an exotic derivative and $\mathbf{P}(c)$ gives its lower no-arbitrage price. A triplet $(f,g,h)$ on the dual side corresponds to a robust sub-hedging strategy for $c$: both $f$ and $g$ are bought through trading European options and $h(x)(y-x)$ corresponds to the payoff from buying $h(x)$ stocks at time zero. The following result establishes the basic duality between the primal and dual problems.
\begin{theorem}[\cite{BeHePe12}]\label{thm:dualitymart}
Let $\mu\preceq\nu\in P_\R$, and assume that $c$ is bounded from below. Then $\mathbf{P}=\mathbf{D}$.
\end{theorem}
In the same paper the authors provided a simple example, based on the cost function $c(x,y)=-|y-x|$, where the dual problem is not attained. Our aim here is to study fundamental reasons why dual attainment may fail and to provide sufficient conditions for it to hold and for dual optimiser to have further desirable regularity and integrability properties. We note that \cite{BeNuTo16} showed dual attainment may be recovered if we weaken the dual formulation and require inequalities to hold quasi-surely, i.e.\ almost surely for any $\P\in \M(\mu,\nu)$.
 However this is not  entirely satisfying in view of the financial and other, mentioned above, applications. 

\section{Main Results}
We start by defining the crucial notion of a solution for the dual problem \eqref{dualproblem}.
\begin{definition}\label{dualmaximizer}
Let $\mu\preceq\nu$ be in convex order and let $c(x,y)$ be a cost function. We say that a triple of  functions $f : \R \to \R \cup \{+\infty\}$, $g : \R \to \R \cup \{-\infty\}$, $h : \R \to \R$ is a dual maximizer, or a solution for the dual problem \eqref{dualproblem}, if
$f$ is finite $\mu$-a.s., $g$ is finite $\nu$-a.s., and 
 for any minimizer $\P^* \in \M(\mu,\nu)$ for the martingale optimal transport problem \eqref{MGTP}, the following holds:
 \begin{align}
\label{dualineq1}g(y) - f(x) - h(x) \cdot (y-x)&\le c(x,y)  \quad  \forall x\in \R, \forall y \in \R, \\
\label{dualeq1}g(y) - f(x) - h(x) \cdot (y-x)&= c(x,y)   \quad \P^*\mbox{-a.a.} \, (x,y).
\end{align}
\end{definition}
{\blue In fact, if \eqref{dualineq1}, \eqref{dualeq1} 
hold for some $\P\in \M(\mu,\nu)$, then $\P$ is a minimizer of the martingale transport problem \eqref{MGTP} and moreover \eqref{dualineq1}, \eqref{dualeq1} hold for all minimizers of \eqref{MGTP}, c.f.\
\cite[Corollary 7.8]{BeNuTo16}. 
 }

A simple but important observation is that if $(f,g,h)$ is a dual maximizer, then $g$ can always be replaced by a ``better" candidate $\tilde g$ induced by $(f,h)$, as follows:
\begin{align}\label{tildeg}
\tilde g(y) := \inf_{x \in \R} \big{(} f(x) +h(x) \cdot (y-x) + c(x,y) \big{)}.
\end{align}
Observe that then $g \le \tilde g$ while \eqref{dualineq1}--\eqref{dualeq1} still holds with $\tilde g$.  The minus signs on $f$ and $h$ in \eqref{dualineq1} and \eqref{dualeq1} were chosen to define $\tilde g$ by \eqref{tildeg}. In this paper, unless stated otherwise we will always assume that $g = \tilde g$. Now we state our main theorem.

\begin{theorem}\label{MainTheorem} Let $\mu\preceq\nu$ be in convex order. 
Suppose that  $c(x,y)$ is semiconvex in $y$ $\mu$-uniformly in $x$, in the following sense:  there exists a Borel function $u : J \to \R$ where $J:={\rm conv}({\rm supp} (\nu))$ such that 
\begin{align}\label{eq:convexity_condition}
\text{for $\mu$-a.e. $x$, $y \mapsto c(x,y) +u(y)$ is continuous and convex on $J$. }
\end{align}
If $\nu$ is not compactly supported, then further suppose that  $y \mapsto c(x,y) +u(y)$ is of linear growth to the direction where $J$ is unbounded.
Then there exists a dual maximizer in the sense of Definition \ref{dualmaximizer}.
\end{theorem}
\begin{corollary}\label{cor:C2}
 In the setting of Theorem \ref{MainTheorem}, if $c\in C^{0,2}$ -- that is $\frac{\partial^2 c}{\partial y^2}$ exists and is continuous on $\Omega$ -- and if $ \nu$ is compactly supported then there exists a dual maximizer.
\end{corollary}
Note that Definition \ref{dualmaximizer} is made in a pointwise sense, that is we do not require $f \in L^1(\mu)$,  $g \in L^1(\nu)$, nor $h \in L^{\infty}(\mu)$. But as already observed in \cite{BeJu16, BeNuTo16}, this classical integrability assumption is  too restrictive for the existence of dual maximizer.  But by using the ``extended notion of integrability" introduced in \cite{BeNuTo16}, this pointwise dual maximizer $(f,g,h)$ may still be viewed as dual maximizer in the generalized sense. To ensure integrability in the classical sense, further assumptions are required, as summarised in the following result.
\begin{theorem}\label{regularity}
Let $\mu\preceq\nu$ be in convex order where $J={\rm conv}({\rm supp} (\nu))$ is compact. Assume that $c$ is Lipschitz on $J \times J$, and there is a Lipschitz  $u : J \to \R$  such that $y \mapsto c(x,y)+u(y)$ is  convex on $J$ for $\mu$-a.e. $x$. Then there exists a dual maximizer $(f,g,h)$ such that $f$ and $g$ are Lipschitz on $J$ and $h$ is bounded on $J$. In particular, $\mu(f) + \nu(g) = \E^{\P^*}[c(X,Y)]$ for any solution $\P^*$ to the problem \eqref{MGTP}.
\end{theorem}
\begin{remark}
In Theorem \ref{regularity} if $c$ and $u$ are Lipschitz with  Lipschitz constant $L$, then $f$ and $g$ can be taken to be Lipschitz with constant $7L$ and $5L$ respectively on $J$, while $|h|$ is bounded by $6L$ on $J$, as shown in the proof.
\end{remark}
We close this section with a discussion of how, and in what sense, the above results are sharp. Examples which support this discussion are presented after the proofs in Section \ref{sec:examples}.
First, we note that the linear growth condition in Theorem \ref{MainTheorem} can not be removed. Indeed, Example \ref{linear} gives a cost function which violates the linear growth condition together with marginals $\mu\preceq\nu$ for which the dual maximizers fails to exist. Second, the convexity condition \eqref{eq:convexity_condition} on $J$ can not be relaxed to just local convexity around $x$, as shown in  Example \ref{ex:toolittleconvexity}, and this even for compactly supported marginals. Third, the $C^{0,2}$ regularity in Corollary \ref{cor:C2}, is optimal in the sense that for any given $\epsilon >0$ we can construct a cost function $c \in C^{2-\epsilon}$ and compactly supported, convex--ordered marginals $\mu\preceq \nu$ for which a dual maximizer satisfying \eqref{dualineq1}--\eqref{dualeq1} does not exist. This is carried out in Example \ref{ex:C^r}. Finally, in Example \ref{nonintegrability} we show the necessity of semiconvexity for the regularity in Theorem \ref{regularity} by showing that there exist  $1$-Lipschitz cost $c$ for which there is a dual maximizer $(f,g,h)$ but $g \notin L^1(\nu)$, even when $(\mu,\nu)$  are compactly supported and irreducible (see Definition \ref{def:irreducible} below). 
\section{Proofs}
To establish Theorem \ref{MainTheorem}, we prove Propositions \ref{finitedomain} and \ref{halfinfinitedomain} below. The key idea is to consider the martingale optimal transport problem on its irreducible components (see Definition \ref{def:irreducible} below and \cite[Appendix A]{BeJu16}). It is known, see Theorem \ref{dualexist} below, that on each irreducible component the dual problem admits a maximizer. Using the semiconvexity assumption on the cost function, we can show that these maximizers are appropriately bounded, such that it is possible to glue them together to obtain global maximizers of the dual problem.\\

Let $\mu, \nu$ be probability measures on $\R$ which are in convex order. It was shown in \cite{BeJu16} and \cite{BeNuTo16}, see Proposition \ref{pr:decomp} below, that there is a canonical decomposition of $\mu, \nu$ into irreducible pairs $(\mu_i, \nu_i)_{i \in \N}$ such that for each irreducible pair $(\mu_i, \nu_i)$, the dual problem attains a solution. For a probability measure $\mu$ on $\R$, we define its potential function by
$$
  u_{\mu}: \R\to\R,\quad   u_{\mu}(x) := \int |x-y|\, d\mu(y).
$$

\begin{definition} \label{def:irreducible}
Let $\mu\preceq\nu$ be in convex order and let $I:=\{x \,: \, u_{\mu}(x) < u_{\nu}(x)\}$. We say that $(\mu,\nu)$ is irreducible on the domain $I$ if $I$ is an open interval and $\mu$ is concentrated on $I$. \end{definition}

We recall the following result from \cite{BeNuTo16} (which requires our standing assumption that $P(c)\in \R$ and  $c(x,y)\le a(x) + b(y)$ for some $a\in L^1(\mu)$ and $b\in L^1(\nu)$.
\begin{theorem}\cite[Theorem 6.2]{BeNuTo16}\label{dualexist} 
 Let $\mu\preceq\nu$ be irreducible on the domain $I$. %and assume that $P(c)$ is finite. 
Then a dual maximizer exists.

\end{theorem}
{\blue
\begin{remark}\label{PrimalAttained}[Existence of a primal minimizer] 
If $c$ is bounded from below, then standard arguments imply the existence of a primal minimizer of the martingale transport problem \eqref{MGTP} i.e.\ $\mathbf{P}=\E^{\P^*}[c(X,Y)]$ for some $\P^*\in\M(\mu,\nu),$, see \cite{BeHePe12}. 
We note that this remains valid in the present setup where $c$ is lower-semicontinuous,   $c$ is  bounded from above in the sense
$c(x,y)\le a(x) + b(y)$ for some $a\in L^1(\mu)$ and $b\in L^1(\nu)$ and $P(c)\in \R$, but $c$ is not necessarily bounded from below. To establish this it is sufficient to argue on irreducible components and on each such component the result is a consequence of \cite[Theorem 6.2]{BeNuTo16} and \cite[Remark 7.9]{BeNuTo16} 
\end{remark}
}

The next proposition claims that for any  dual maximizer $(f,g,h)$ in Theorem \ref{dualexist}, if the function $y \mapsto c(x,y)$ is convex for each $x \in I$, then the lower envelope function $g$  has a ``desirable shape" modulo an affine function. We first deal with the bounded domain case.

\begin{proposition}\label{finitedomain}
Let $\mu\preceq\nu$ be irreducible on a bounded  domain $I=]a,b[$ and assume that $(f,g,h)$ is a dual maximizer  in Theorem \ref{dualexist}.  Suppose that there exist $A \in \R \cup \{-\infty\}$, $B \in \R \cup \{\infty\}$ such that $A \le a<b\le B$ and that for $\mu$-a.e.\ $x$,
\begin{align}
\text{$y \mapsto c(x,y)$ is continuous and convex on $[A,B]$.}
\end{align}

Then  we can find an affine function $L(y) = L(x) + \nabla L \cdot (y-x)$ such that $\big{(}\tilde f (x), \tilde g(y), \tilde h(x)\big{)} := \big{(}f(x) - L(x), g(y) - L(y), h(x) - \nabla L\big{)}$ is a dual maximizer, and furthermore

\begin{align}
\label{negative}&\tilde g(y) \le 0 \text{ on } ]a,b[,\\
\label{positive}&\tilde g(y) \ge 0 \text{ on }[A, a] \cup [b, B],\\
\label{zero}&%\text{If  $\exists x_a, (x_a, a) \in G$
\text{If $\nu(a) >0$ then $\tilde g(a)=0$. If $\nu(b) >0$ then $\tilde g (b) =0$.}
\end{align}

\end{proposition}

\begin{proof} 
{\bf Step 1.} Let us begin by recalling some terminology from \cite{BeJu16}: for a set $\Gamma \subset \R \times \R$, denote $X_\Gamma$ as its projection to the first coordinate space $\R$, and $Y_G$ to the second. We will also write $\Gamma_x=\{y: (x,y)\in \Gamma\}$.

Now let $\Gamma \subset \R \times \R$ be the ``contact set" induced by the dual optimizer $(f,g,h)$, that is 
\begin{align}
\Gamma:= \{ (x,y) \,:\, g(y) - f(x) - h(x) \cdot (y-x)= c(x,y) \}.
\end{align}
Then by definition every solution $\P^*$ to the primal problem \eqref{MGTP} is concentrated on $\Gamma$, i.e. $\P^*(\Gamma)=1$. Since every such $\P^*$ is a martingale measure,  we can  find a subset $G$ of $\Gamma$ such that every solution to \eqref{MGTP} is still concentrated on $G$, $Y_G \subset [a,b]$ by the irreducibility of $(\mu, \nu)$, and furthermore $G$ is {\it regular} in the following sense:
\begin{align}
\text{For each } x\in X_G, \text{ either} \quad  x \in {\rm int(conv} (G_x)) \quad \text{or} \quad G_x = \{x\}.
\end{align}
Hence, we have
\begin{align}
\label{dualineq}g(y) - f(x) - h(x) \cdot (y-x)&\le c(x,y)  \quad  \forall x\in \R, \forall y \in \R, \\
\label{dualeq}g(y) - f(x) - h(x) \cdot (y-x)&= c(x,y)   \quad \text{on}\quad G
\end{align}
where $f,g$ are real-valued on $X_G,Y_G$ respectively, $Y_G \subset [a,b]$, $G$ is regular, $\P^*(G)=1$ for every solution $\P^*$ to \eqref{MGTP}, and $y \mapsto c(x,y)$ is convex on $[A,B]$ for every $x \in X_G$. Now for each $x \in X_G$, let $I(x) = {\rm int(conv} (G_x))$ which is an open interval or a point $\{x\}$ if $G_x = \{x\}$. As $\P^*(G)=1$ and $(\mu, \nu)$ is irreducible, we can find a sequence $\{x_n\} \subset X_G$ such that each $I(x_n)$ is an open interval (and we write $I(x_n)=]a_n, b_n[$), and defining  $l_n = \min_{i \le n} a_i$, $r_n = \max_{i \le n} b_i$ and  $J_n = ]l_n,r_n[$, we have
\begin{align}
&J_n \cap I(x_{n+1}) \neq \emptyset  \quad \forall n, \quad\text{and}\\
&l_n \searrow a, r_n \nearrow b \quad \text{as} \quad  n \to \infty.
\end{align}
(For more details about the existence of such a sequence, we refer to \cite[Appendix A]{BeJu16}.) We may assume that the case $l_n < a_{n+1} < b_{n+1} < r_n$ does not occur, since if it occurs then we may simply discard the respective $x_{n+1}$ from the sequence.\\
\
Note that as $y \mapsto c(x,y)$ is continuous and thus $g(y)$ is upper semicontinuous (recall \eqref{tildeg}), \eqref{dualeq} holds for $(x_n, a_n)$ and $(x_n, b_n)$ for every $n$. Hence
\begin{align}\label{endpoints}
\text{For each $n$, there exist $1 \le i,j \le n$ such that \eqref{dualeq} holds for $(x_i, l_n)$ and $(x_j, r_n)$.}
\end{align}
{\bf Step 2.} For each $n$, define
\begin{align}\label{defofv}
v_{x_n}(y) &=  c(x_n,y) + f(x_n) + h(x_n)\cdot(y-x_n),\\
\label{defofg}g_n(y) &= \inf_{ i \le n}  \big{(}v_{x_i}(y)\big{)},\\
\label{defofL}L_n(y) &= \frac{g_n(r_n) - g_n(l_n)}{r_n - l_n}\cdot (y - l_n) + g_n(l_n).
\end{align}
We claim that, for each $n$,
\begin{align}
\label{nnegative}&g_n(y) \le L_n(y) \quad \forall y \in [l_n,r_n], \\
\label{npositive}&g_n(y) \ge L_n(y) \quad \forall y \in [A, l_n] \cup [r_n, B].
\end{align}
The claim is obvious for $n=1$ since $v_{x_1}(y)$ is convex. Suppose that the claim is true for $n$. Then there are two cases.\smallskip\\
\noindent{\it Case 1 : $a_{n+1} \le l_n < r_n \le b_{n+1}$.}\\
First, since $g_{n+1} (y) = \min \big{(}g_n(y), v_{x_{n+1}}(y)\big{)}$ and since $g_{n+1} (a_{n+1}) = v_{x_{n+1}}(a_{n+1})$, $g_{n+1} (b_{n+1}) = v_{x_{n+1}}(b_{n+1})$ by \eqref{dualineq}, \eqref{dualeq}, by convexity of $v_{x_{n+1}}(y)$ we see that 
\begin{align}\label{aa}
g_{n+1}(y) \le v_{x_{n+1}}(y) \le L_{n+1}(y) \quad \forall y \in [a_{n+1}, b_{n+1}].
\end{align}
This establishes the claim in \eqref{nnegative} for $n+1$.\\
For the second claim, fix $i \in \{1,2,\ldots,n+1\}$. Then there exists $y_i \in [a_{n+1}, b_{n+1}]$ such that $(x_i, y_i) \in G$. Then \eqref{aa} implies
\begin{align}\label{ab}
v_{x_{i}}(y_i) = g_{n+1}(y_i) \le L_{n+1} (y_i).
\end{align}
Also note that by \eqref{dualineq}, \eqref{dualeq}, 
\begin{align}\label{ac}
v_{x_{i}}(a_{n+1}) \ge v_{x_{n+1}}(a_{n+1}) = g_{n+1}(a_{n+1}), \quad v_{x_{i}}(b_{n+1}) \ge v_{x_{n+1}}(b_{n+1}) = g_{n+1}(b_{n+1}).
\end{align}
Now by \eqref{ab}, \eqref{ac} and convexity of $v_{x_{i}}(y)$, we deduce that
\begin{align}\label{ad}
v_{x_{i}}(y) \ge L_{n+1}(y) \quad \forall y \in [A, a_{n+1}] \cup [b_{n+1}, B].
\end{align}
As \eqref{ad} holds for every $i$, this verifies the claim \eqref{npositive} for $n+1$, completing the inductive step.\smallskip\\
\noindent{\it Case 2 : $l_n \le  a_{n+1} \le r_n \le b_{n+1}$.}\\
First, since $v_{x_{n+1}}(a_{n+1})=g_{n+1}(a_{n+1}) \le g_{n}(a_{n+1})$, by the induction hypothesis
\begin{align}\label{ae}
v_{x_{n+1}}(a_{n+1})  \le L_n(a_{n+1}).
\end{align}
Also note that by \eqref{endpoints}, we have
\begin{align}\label{af}
v_{x_{n+1}}(r_n) \ge g_n(r_n)=L_n(r_n).
\end{align}
Now the convexity of $v_{x_{n+1}}$ implies 
\begin{align}\label{ag}
v_{x_{n+1}}(b_{n+1}) \ge L_n(b_{n+1}).
\end{align}
Note that $L_n(l_n)=g_n(l_n) = g_{n+1}(l_n)$ and $g_{n+1}(b_{n+1}) = v_{x_{n+1}}(b_{n+1})$. Hence by \eqref{ag},
\begin{align}\label{ah}
 \lambda g_{n+1} (l_n) +  (1-\lambda) g_{n+1} (b_{n+1}) \ge  \lambda L_n (l_n) +  (1-\lambda) L_n (b_{n+1})         \quad \forall \lambda \in [0,1].
\end{align}
As $g_{n+1} (y) = \min \big{(}g_n(y), v_{x_{n+1}}(y)\big{)}$, the induction hypothesis and convexity of $v_{x_{n+1}}(y)$ along with \eqref{ae}, \eqref{ah} imply the first claim \eqref{nnegative} for $n+1$.\\
For the second claim, fix $i \in \{1,2,\ldots,n+1\}$. Then there exists $y_i \in [l_n, b_{n+1}]$ such that $(x_i, y_i) \in G$, and the first claim  \eqref{nnegative} for $n+1$ gives $v_{x_{i}}(y_i) \le  L_{n+1} (y_i)$. On the other hand, $v_{x_{i}}(l_n) \ge g_n(l_n)=g_{n+1}(l_n)$ and $v_{x_{i}}(b_{n+1}) \ge v_{x_{n+1}}(b_{n+1}) = g_{n+1}(b_{n+1})$. Hence by convexity of $v_{x_{i}}(y)$, we deduce that
\begin{align}\label{ai}
v_{x_{i}}(y) \ge L_{n+1}(y) \quad \forall y \in [A, l_n] \cup [b_{n+1}, B].
\end{align}
As \eqref{ai} holds for every $i$, this verifies the claim \eqref{npositive} for $n+1$, completing the induction. The case $a_{n+1} \le l_n \le  b_{n+1} \le r_n$ can be treated in the same way.\smallskip\\
{\bf Step 3.} We claim that
\begin{align}\label{uniformbound}
\text{$L_n(y)$'s are uniformly bounded on $[a,b]$ for all $n$.}
 \end{align}
To prove \eqref{uniformbound}, choose $M>0$ such that $|v_{x_{1}}(y)| \le M$ on $[a,b]$. Then $L_n(l_n)=g_n(l_n) \le v_{x_{1}}(l_n) \le M$ and $L_n(r_n)=g_n(r_n) \le v_{x_{1}}(r_n) \le M$. Hence, as $L$ is linear, $L_n(a_1) \le M$ and $L_n(b_1) \le M$. On the other hand, by Step 2, $-M \le v_{x_{1}}(a_1) =g_n(a_1) \le L_n(a_1)$ and $-M \le v_{x_{1}}(b_1) =g_n(b_1) \le L_n(b_1)$. This implies  \eqref{uniformbound}. In particular, there exists a subsequence of $L_n$ (which we denote as $L_k$) such that $L_k(y)$ uniformly converges to an affine function as $k \to \infty$, say $L (y)$ on every compact interval in $\R$. Now we claim that, for $v_x(y) :=  c(x,y) + f(x) + h(x)\cdot(y-x)$ and $g(y) := \displaystyle\inf_{x \in X_G}\big{(}v_x(y)\big{)}$,
\begin{align}
\label{lnegative}&g(y) \le L(y) \text{ on } ]a,b[,\\
\label{lpositive}&g(y) \ge L(y) \text{ on } [A, a] \cup [b, B].
\end{align}
First it is easy to see \eqref{lnegative} as follows: if $y \in ]a,b[$ then for all large $k$ we have $y \in ]l_k, r_k[$, thus by Step 2, $g(y) \le g_k(y) \le L_k(y)$. By taking $k \to \infty$, we see that $g(y) \le L(y)$, proving \eqref{lnegative}.

Next, suppose that there exists  $(x,y) \in G$ with $a<y<b$. Then again for all large $k$ we have $v_x(y) = g(y) \le g_k(y) \le L_k(y)$, thus $v_x(y) \le L(y)$. On the other hand, by \eqref{dualineq}, \eqref{dualeq} we have $v_x(l_k) \ge g_k(l_k)$ and $v_x(r_k) \ge g_k(r_k)$, thus $v_x(a) \ge L(a)$ and $v_x(b) \ge L(b)$ by letting $k \to \infty$. By convexity of $v_x$, this implies that $v_x(y) \ge L(y) \text{ on } [A, a] \cup [b, B]$.

If there is no $y$ such that $a<y<b$ and  $(x,y) \in G$, this means that $G_x = \{a,b\}$, i.e.\ $(x,a), (x,b) \in G$. Then without loss of generality we may simply include this $x$ in the sequence $\{x_n\}$ defined in Step 1, say we put  $x=x_1$. This implies that $l_n = a$ and $r_n=b$ for all $n$. Thus $v_x(a) = L(a)$ and $v_x(b) = L(b)$. By convexity of $v_x$, this implies that $v_x(y) \ge L(y) \text{ on } [A, a] \cup [b, B]$. Hence, for any $x\in X_G$ we deduce that $v_x(y) \ge L(y) \text{ on } [A, a] \cup [b, B]$, therefore \eqref{lpositive} follows.

If $\nu(a)>0$ then there exists $x_a \in I$ such that $(x_a, a) \in G$. Then again we may include $x_a$ in the sequence $\{x_n\}$ defined in Step 1. This implies that $l_n = a$ for all large $n$, thus $g(a) = v_{x_a}(a) = L(a)$. Similarly if $\nu(b) >0$  then $g(b)  = L(b)$.

Finally, we see that  $\big{(}\tilde f (x), \tilde g(y), \tilde h(x)\big{)} := \big{(}f(x) - L(x), g(y) - L(y), h(x) - \nabla L\big{)}$ satisfies \eqref{dualineq}, \eqref{dualeq}, \eqref{negative}, \eqref{positive} and \eqref{zero}, concluding the proof.
\end{proof}

Next, we deal with the half-infinite domain case.

\begin{proposition}\label{halfinfinitedomain}
Let $\mu\preceq\nu$ be irreducible on a half-infinite  domain $I=]0, \infty[$ and assume that $(f,g,h)$ is a dual maximizer  in Theorem \ref{dualexist}. Suppose that there exists $A \in [-\infty, 0]$ such that for $\mu$-a.e.\ $x$,
\begin{align*}
&\text{$y \mapsto c(x,y)$ is continuous and convex on $[A, \infty[$, and }\\
&\text{there exists an affine function $L_x$ such that } c(x,y)  \le L_x(y) \text{ for all } y\ge 0.
\end{align*}
Then  we can find an affine function $L$ such that $\big{(}\tilde f (x), \tilde g(y), \tilde h(x)\big{)} := \big{(}f(x) - L(x), g(y) - L(y), h(x) - \nabla L\big{)}$ is a dual maximizer, and furthermore
\begin{align}
\label{negative1}&\tilde g(y) \le 0 \text{ on } ]0,\infty[,\\
\label{positive1}&\tilde g(y) \ge 0 \text{ on } [A, 0],\\
\label{zero1}&\text{if $\nu(0) >0$ then $\tilde g(0)=0$.}
\end{align}

\end{proposition}

\begin{proof} 
{\bf Step 1.}  Let $G \subset \R \times \R$ be chosen as in the Step 1 in Proposition \ref{finitedomain}, i.e. $Y_G \subset [0,\infty[$, $G$ is regular, $\P^*(G)=1$ for every solution $\P^*$ to \eqref{MGTP}, $y \mapsto c(x,y)$ is convex and bounded above by an affine function on $[0,\infty[$ for every $x \in X_G$, and \eqref{dualineq}, \eqref{dualeq} holds. For each $x \in X_G$, let $I(x) = {\rm int(conv} (G_x))$ as before. We can find a sequence $\{x_n\} \subset X_G$ such that each $I(x_n)$ is an open interval (and we write $I(x_n)=]a_n, b_n[$, where $b_n$ can be $+\infty$), and if we define $l_n = \min_{i \le n} a_i$, $r_n = \max_{i \le n} b_i$ and also define $J_n = ]l_n,r_n[$, then
\begin{align}
&J_n \cap I(x_{n+1}) \neq \emptyset  \quad \forall n, \quad\text{and}\\
&l_n \searrow 0, r_n \nearrow +\infty \quad \text{as} \quad  n \to \infty.
\end{align}
\smallskip\\
{\bf Step 2.} Recall definitions \eqref{defofv} -- \eqref{defofL}. Altering   the triple $(f,g,h)$ by an appropriate affine function and using the condition of linear growth and convexity satisfied by the cost, we can assume that 
\begin{align}\label{asymptote}
\text{$v_{x_{1}}(y)$ is decreasing on $[A, \infty[$, $ v(0)=0$ and $\lim_{y \to \infty} v_{x_{1}}(y) = b > -\infty$. }
\end{align}
Now we claim that, for each $n$,
\begin{align}\label{horizon}
g_n(y) \text{ is decreasing on } [A, l_n], \text{ and } g_n(y) \le g_n(l_n) \text{ on } [l_n, \infty[.
\end{align}
Note that the claim \eqref{horizon} is obviously true for $n=1$ by the assumption \eqref{asymptote}. Suppose the claim is true for $n$. We will show that the claim is also true for $n+1$. To see this, note that as $b_{n+1} \ge l_n$, using \eqref{dualineq}, \eqref{dualeq} and the induction hypothesis \eqref{horizon}, we see that
\begin{align}\label{ba}
v_{x_{n+1}}(b_{n+1}) \le g_n (b_{n+1}) \le g_n(l_n), \text{ while } v_{x_{n+1}}(l_n) \ge g_n(l_n).
\end{align}
(If $b_{n+1} = \infty$, then instead of $b_{n+1}$ we may argue with arbitrarily large $c_{n+1}$ satisfying $(x_{n+1}, c_{n+1}) \in G$.)
By \eqref{ba} and convexity of $v_{x_{n+1}}(y)$, we see that $v_{x_{n+1}}(y)$ is decreasing on $[A, l_n]$. As $g_{n+1}(y) = \min \big{(}g_n(y), v_{x_{n+1}}(y) \big{)}$ and  $g_n(y)$ is decreasing on $[A, l_n]$, we see that 
\begin{align}\label{}
g_{n+1}(y) \text{ is decreasing on } [A, l_n].
\end{align}
In particular, for any $y \in [l_{n+1}, l_n]$ we have $g_{n+1}(l_{n+1}) \ge g_{n+1}(y)$. For $y \ge l_n$, we see that $g_{n+1}(l_{n+1}) \ge g_{n+1}(l_{n}) = g_{n}(l_{n})$ by \eqref{endpoints}, and $g_{n}(l_{n}) \ge g_{n}(y) \ge g_{n+1}(y)$ by \eqref{horizon}. Hence
\begin{align}\label{}
g_{n+1}(y) \le g_{n+1}(l_{n+1}) \text{ on } [l_{n+1}, \infty[.
\end{align}
Therefore, \eqref{horizon} is proved for all $n$.\\
We have observed that the sequence $\{g_{n}(l_{n})\}$ is increasing. Note that $g_{n}(l_{n}) \le v_{x_{1}}(l_{n}) \le v_{x_{1}}(0)= 0$ for all $n$, thus $\{g_{n}(l_{n})\}$ converges to, say, $L$. Then we claim
\begin{align}
\label{lnegative1}&g(y) \ge L \text{ on } [A,0],\\
\label{lpositive2}&g(y) \le L \text{ on } ]0, +\infty[,
\end{align}
where, as before, $v_x(y) := c(x,y) + f(x) + h(x)\cdot(y-x)$  and $g(y) := \displaystyle\inf_{x \in X_G}\big{(}v_x(y)\big{)}$. To see this, fix $x>0$. Then there exists $n$ such that $l_n < x$. Arguing as above, we see that $v_x(y)$ is decreasing on $[A, l_n]$ and $v_x(l_n) \ge g_n(l_n)$ for all $n$. Hence for any $y \le 0$ we see that $v_x(y) \ge g_n(l_n)$. Letting  $n \to \infty$ we conclude
\begin{align}\label{}
g(y) \ge L \quad \text{for all} \,\,y \in [A,0].
\end{align}
Now for $y >0$, there exists $n$ such that $l_n < y$. Then by \eqref{horizon}, we see that $g(y) \le g_n(y) \le g_n(l_n)$ for all large $n$, thus by taking $n \to \infty$ we conclude
\begin{align}\label{}
g(y) \le L \quad \text{for all} \,\,y > 0.
\end{align}
If $\nu(0) >0$ then there is $x \in X_G$ with $(x,0) \in G$, and we may simply put this $x$ into the sequence $\{x_n\}$ by letting $x=x_1$. Then every $l_n$ simply becomes $0$ and $\{g_n(l_n)\}$ becomes the constant  sequence $L$. Hence, $g(0)=L$. Finally, altering the triple $(f,g,h)$ by the constant function $-L$, we can assume that $L=0$. This proves the proposition.
\end{proof}

We are now ready to show the existence of dual optimizers for martingale optimal transport problem in Theorem \ref{MainTheorem}. In particular we no longer assume the irreducibility of $(\mu, \nu)$.
Note that If $(\mu, \nu)$ is irreducible on the domain $I = \R$ then Theorem \ref{MainTheorem} simply follows from Theorem \ref{dualexist}. Otherwise, $(\mu, \nu)$ can be decomposed into at most countably many irreducible components, and any martingale $\P \in \rm{MT} (\mu, \nu)$ is decomposed accordingly. More precisely we recall:
%The following proposition is taken from \cite{BeNuTo16}. % and its proof is given in \cite{BeJu16}.
\begin{proposition} \cite[Theorem A.4]{BeJu16}\label{pr:decomp}
Let $\mu\preceq\nu$ and let $(I_{k})_{k \ge 1}$ be the open connected components of the set $\{x \,:\, u_{\mu}(x)<u_{\nu}(x)\}$. Set $I_{0}=\R\setminus \cup_{k\geq1} I_{k}$ and $\mu_{k}=\mu|_{I_{k}}$ for $k\geq 0$, so that $\mu=\sum_{k\geq0} \mu_{k}$.
  Then, there exists a unique decomposition $\nu=\sum_{k\geq0} \nu_{k}$ such that
  $$
    \mu_{0} = \nu_{0} \quad\quad \mbox{and}\quad\quad \mu_{k} \preceq \nu_{k} \quad \mbox{for all} \quad k\geq1,
  $$
  and this decomposition satisfies $I_{k}=\{x \,:\, u_{\mu_{k}}(x) < u_{\nu_{k}}(x)\}$ for all $k\geq1$.
  Moreover, any $\P\in \rm{MT}(\mu,\nu)$ admits a unique decomposition
  $
    \P=\sum_{k\geq0} \P_{k}
  $
  such that $\P_{k}\in \rm{MT}(\mu_{k},\nu_{k})$ for all $k\geq0$.
\end{proposition}

Note that as  $\mu_{0} = \nu_{0}$, $\P_0$ must be the identity martingale. We can now give the proof of our  first main result. 

\begin{proof}[Proof of Theorem \ref{MainTheorem}] 
Notice that by definition of the dual maximizer and the assumption on the cost, we can assume that $y \mapsto c(x,y)$ is continuous and convex on $J:={\rm conv}({\rm supp} (\nu))$. Let $\P^*$ be any minimizer in $\rm{MT} (\mu, \nu)$ for the problem \eqref{MGTP}. Then $\P^*_k$ is a minimizer in $\rm{MT} (\mu_k, \nu_k)$. For each $k \ge 1$, choose a set $G_k \subset \R \times \R$ and a triple $(f_k, g_k, h_k)$ provided by Proposition \ref{finitedomain} if $I_k$ is bounded, or by Proposition \ref{halfinfinitedomain} if $I_k$ is half-infinite. We need to define $G_0$ and $(f_0, g_0, h_0)$ for $I_0$. As $\P^*_0$ is the identity map, of course we take $G_0 := \{(x,x) \,:\, x \in I_0\}$. For each $x \in I_0$ define $f_0(x) = -c(x,x)$, and choose $h_0(x)$ in such a way that the convex function $v_x(y) := c(x,y) +f_0(x) + h_0(x) \cdot (y-x)$ satisfies $v(x) = v ' (x) = 0$ (more precisely $0$ belongs to the subdifferential of $v_x$ at $x$). Define $g_0 (y) = \inf_{x \in I_0}\{ v_x (y)\}$ so in particular $g_0 = 0$ on $I_0$. Finally, define
\begin{align}
f(x) &= f_k(x) \quad \text{if} \quad x \in X_{G_k},\\
h(x) &= h_k(x) \quad \text{if} \quad x \in X_{G_k},\\
g(y) &= \inf_{k \ge 0} g_k (y).
\end{align}
Let $G = \cup_{k \ge 0} G_k$. Obviously $\P^*(G)=1$. Now observe that the properties \eqref{negative}, \eqref{positive}, \eqref{zero}, \eqref{negative1}, \eqref{positive1}, \eqref{zero1} verified in Proposition \ref{finitedomain}, \ref{halfinfinitedomain} clearly indicate that the triples $(f_k, g_k, h_k)_{k \ge 0}$ are compatible, that is, the duality \eqref{dualineq}, \eqref{dualeq} holds for $G$ and $(f,g,h)$. This completes the proof.
\end{proof}

\begin{remark} The linear growth assumption of the function $y \mapsto c(x,y)$ is required only for those $x$ in the half-infinite irreducible domain of $\mu, \nu$ as in Proposition \ref{halfinfinitedomain}. 
\end{remark}
%Finally, we give the proof of Theorem \ref{regularity}.
\begin{proof}[Proof of Theorem \ref{regularity}] 
 We will say that a function $f$ is $L$-Lipschitz on $D$ if $|f(x)-f(y)| \le L|x-y|$ for all $x,y \in D$. Assume $c$ is $L_1$-Lipschitz on $J\times J$ and $u$ is $L_2$-Lipschitz on $J$. 

Consider the decomposition of $(\mu,\nu)$ into irreducible pairs given by Proposition \ref{pr:decomp}. Fix $i \ge 1$ and consider $(\mu_i,\nu_i)$ which is irreducible on the bounded domain $I_i=]a_i,b_i[$. Let $\tilde c(x,y) := c(x,y) + u(y)$, and let $(f_i,g_i,h_i)$ be a dual maximizer for the cost $\tilde c$ and $(\mu_i,\nu_i)$, satisfying the conclusion of Proposition \ref{finitedomain}.  In the Step 1 of the proof of Proposition \ref{finitedomain} we explained that there exists a regular set $G_i \subset I_i \times \overline I_i$ on which every (decomposed) solution $\P^*_i$ to the problem \eqref{MGTP} is concentrated, so that the duality \eqref{dualineq1}, \eqref{dualeq1} holds with $(f_i,g_i,h_i)$, $\tilde c$, and $G_i$.

Define $v_{i,x}(y) := f_i(x)+h_i(x)\cdot(y-x) + \tilde c(x,y)$  and note that as $y \mapsto \tilde c(x,y)$ is $L$-Lipschitz and convex on $J= {\rm conv}(\supp \nu) = [A,B]$ where $L= L_1+L_2$, we have
\begin{align}\label{derivativedifference}
\frac{d v_{i,x}}{dy}(b_i^-)  -\frac{d v_{i,x}}{dy}(a_i^+)  \le 2L.
\end{align}
Proposition \ref{finitedomain} tells us that $g_i(a_i) \ge 0$, $g_i(b_i) \ge 0$ while there is $y \in [a_i,b_i]$ such that $g_i(y) \le 0$, since $g_i(y) \le 0$ whenever $(x,y) \in G_i$ for some $x$. As $g_i(y) := \inf_{x \in X_{G_i}} \{v_{i,x}(y)\}$, if $(x,y) \in G_i$ then we have $v_{i,x}(a_i) \ge 0$, $v_{i,x}(b_i) \ge 0$ while $v_{i,x}(y) \le 0$. With \eqref{derivativedifference} this implies that $v_{i,x}$ is $2L$-Lipschitz for any $x\in X_{G_i}$, hence $g_i$ is also   $2L$-Lipschitz on $[a_i,b_i]$. Proposition \ref{finitedomain} then tells us that we can  replace $g_i$ with $\tilde g_i:=g_i \wedge 0$, so that $\tilde g_i$ is $2L$-Lipschitz on $\R$, $\tilde g_i \le 0$ on $[a_i,b_i]$,  $\tilde g_i=0$ on $\R \setminus ]a_i,b_i[$, and $(f_i, \tilde g_i, h_i)$ satistfies \eqref{dualineq1} -- \eqref{dualeq1}.

In the proof of Theorem \ref{MainTheorem} we showed that there is a dual maximizer $(f,g,h)$ to the problem \eqref{MGTP} where $ g:= \inf_{i \ge 0} \tilde g_i$ (recall that $\tilde g_0 \equiv 0$ on $I_0$). Hence we get that $ g$ is $2L$-Lipschitz on $\R$, $ g\le 0$ on $J$, and  $ g = 0 $ on $J^c$.

Next, observing the duality relation
\begin{align}
\label{dualineq2}  g(y)  - \tilde c(x,y)  \le f(x) + h(x) \cdot (y-x)  \quad  \forall x\in J, \forall y \in J,
\end{align}
note that we can replace $f$, $h$ by $\tilde f$, $\tilde h$ respectively, as follows: define $H:J\times J \to \R$ by the upper concave envelope in $y$ variable
$$H(x,y) := {\rm conc}[g(\cdot) - \tilde c(x, \cdot)](y).$$
Then we define $\tilde f(x) := H(x,x)$ and $\tilde h(x) := \frac{\partial H(x,y)}{\partial y} \big{|}_{y=x}$. More precisely, $\tilde h(x)$ is an element of the superdifferential of the concave function $y \to H(x,y)$ at $x$, and there exists a measurable choice of such an $\tilde h$. Now in view of \eqref{dualineq2}, it is clear that $(\tilde f,  g, \tilde h)$ is a dual maximizer. Observe that since $y \mapsto g(y) - \tilde c(x,y)$ is $3L$-Lipschitz, it is immediate that $|\tilde h| \le 3L$ on $J$. Then, to see that $\tilde f$ is Lipschitz, note that since $x \mapsto g(y) - \tilde c(x,y)$ is $L_1$-Lipschitz, by definition of $H$ we have
\begin{align*}
| H(x,y) - H(x',y)| \le L_1|x-x'| \quad \forall x, x', y \in J.
\end{align*}
On the other hand, since the concave envelope of a Lipschitz function is Lipschitz,
\begin{align*}
| H(x,y) - H(x,y')| \le 3L|y-y'| \quad \forall x, y, y' \in J.
\end{align*}
These inequalities immediately  imply that, for any $x,x' \in J$,
\begin{align*}
|\tilde f (x) - \tilde f(x')| = | H(x,x) - H(x',x')| \le (L_1+3L)|x-x'|.
\end{align*}
Finally, recall that $ \tilde c(x,y) = c(x,y) + u(y)$ where $u$ is $L_2$-Lipschitz, we replace $g$ with $\tilde g = g-u$ which is $2L+L_2$ Lipschitz on $J$. Then $(\tilde f, \tilde g, \tilde h)$ is a dual maximizer satisfying the conclusion of Theorem \ref{regularity}.
\end{proof}

\section{Some insightful examples}\label{sec:examples}
\begin{example}\label{linear} In this example, we show that the linear growth condition on the cost for the half-infinite domain case cannot be dropped in Theorem \ref{MainTheorem}. \\
Let $x_n = n$, $n=1,2,3,\ldots$ and let $G_{n} = \{n-1, n+1\}$. For each $n$, choose $c(n,y)$, $f(n)$ and $h(n)$ appropriately such that  $v_n(y) =  c(n,y) + f(n) + h(n)\cdot(y-n)$ becomes
\begin{align}\label{}
&v_n(y) = y^2 \quad\quad\quad\,\,\,\, \text{ if } y \ge n-1,\\
&v_n(y) = (n-1)y \quad  \text{ if } y \le n-1.
\end{align}
Then the triple $(f,g,h)$ supports the set $G = \{(n, n+1), (n, n-1) \,:\, n \in \N\}$  in view of \eqref{dualineq}, \eqref{dualeq}, and clearly 
\begin{align}\label{}
g(y) = -\infty \quad \text{on} \quad  ]-\infty, 0[.
\end{align}
Hence, Proposition \ref{halfinfinitedomain} cannot hold in this case.

\end{example}

\begin{example} \label{ex:toolittleconvexity}
In this example, we show that if the convexity assumption on $y \mapsto c(x,y)$ holds only locally around $x$, then the dual maximizer can fail to exist.

Define the cost function by
\begin{align}
c(x_\infty, y) = 0, \quad\text{and} \quad c(x_n,y) = 
    \begin{cases}
    0 \quad \text{if} \quad y \in [y_{n-1}, y_n], \\
    y - y_{n-1} \quad \text{if} \quad y \le y_{n-1},\\
    -y + y_n \quad \text{if} \quad y \ge y_n.
    \end{cases}
\end{align}
Let $y_0=0$, $y_n=\sum_{k=1}^n \frac{1}{n^2}$, $x_n = (y_{n-1} + y_n) / 2$, and $x_\infty = y_\infty = \sum_{k=1}^\infty \frac{1}{n^2}$. Let $\mu$ be any probability measure whose support is $\{x_n\}_{1 \le n \le \infty}$, and construct a martingale measure $\P$ whose disintegration $(\P_x)_x$ w.r.t. $\mu$ is as follows:
$$\P_{x_n} = \frac{1}{2}(\delta_{y_{n-1}} + \delta_{y_n}) \quad \forall n=1,2,\ldots\quad\text{and}\quad \P_{x_\infty} = \delta_{x_\infty},$$
and define $\nu$ as the second marginal of $\P$. Notice that then $\P$ is the unique element in ${\rm MT}(\mu, \nu)$.  Now we will show that this (optimal) $\P$ does not allow a dual maximizer, that is, there does not exist a triple $(f,g,h)$ which satisfies the following:
\begin{align}
\label{dualineq3}g(y) &\le c(x_n,y) + f(x_n) + h(x_n) \cdot (y-x_n) \quad \forall n \in \N \cup \{\infty\}, \,\,\forall y \in \R, \\
\label{dualeq3+}g(y_n) &= c(x_n,y_n) + f(x_n) + h(x_n) \cdot (y_n -x_n) \quad  \forall n \ge 1,\\
\label{dualeq3-} g(y_{n-1}) &= c(x_n,y_{n-1}) + f(x_n) + h(x_n) \cdot (y_{n-1} -x_n)  \quad \forall n \ge 1,\\
\label{dualeq30} g(y_\infty) &= c(x_\infty,y_\infty) + f(x_\infty) + h(x_\infty) \cdot (y_\infty-x_\infty) = f(x_\infty).
\end{align}
Recall that once such a $(f,g,h)$ exists, then we can redefine $g$ as follows:
\begin{align}
\label{newg}g(y) := \inf_{n \in \N \cup \{\infty\}} \big{(} c(x_n,y) + f(x_n) + h(x_n) \cdot (y-x_n) \big{)}.
\end{align}
We claim that, if we have such a $(f,g,h)$, then we must have
$$g(y_\infty) = -\infty,$$
 which is a contradiction to \eqref{dualeq30}. To see this, for convenience let us define 
 \begin{align}\label{v_x}
v_x(y) := c(x,y) + f(x) + h(x) \cdot (y-x).
\end{align}
 Then  by \eqref{dualineq3}, \eqref{dualeq3+}, \eqref{dualeq3-} we must have
 \begin{align}\label{v_xproperty1}
 v_{x_n} (y_n) &= v_{x_{n+1}} (y_n), \quad\text{and}\\
\label{v_xproperty2} v_{x_n} (y_{n-1}) &\le v_{x_{n+1}} (y_{n-1}),  \quad \forall n \ge 1.
\end{align}
Notice that these with the definition of $c(x_n, y)$ immediately implies 
\begin{align}\label{slopedecrease}
h(x_n) \ge h(x_{n+1}) +1, \quad \forall n \ge 1.
\end{align}
Also notice that  $g(y)$ is a piecewise linear function on $[0, y_\infty[$, and in fact $g(y) = f(x_n) + h(x_n) \cdot (y-x_n)$ on $[y_{n-1}, y_n]$. Hence by \eqref{slopedecrease} and the fact $\sum_n \frac{1}{n} = \infty$ and the concavity of $g$, we see that 
\begin{align*}%\label{slopedecrease}
 g(y_\infty) = \lim_{y \nearrow y_\infty} g (y) = - \infty,
\end{align*}
a contradiction to the fact that $g(y_\infty)$ must be real-valued.
\end{example}

\begin{example}\label{ex:C^r} In this example, we show that the $C^2$ regularity required in Theorem \ref{MainTheorem} is optimal in the following sense: for any $1<r<2$, we construct a cost function $c \in C^r$ and compactly supported marginals $\mu\preceq\nu$ for which the dual attainment fails. This example shall be a slight modification of the previous one. First, let 
\begin{align}c(x,y) = -|x-y|^r,
\end{align}
and choose $s$ such that 
\begin{align}s>1 \quad \text{and} \quad sr <2.
\end{align}
Let $y_0=0$, $y_n=\sum_{k=1}^n \frac{1}{n^s}$, $x_n = (y_{n-1} + y_n) / 2$, and $x_\infty = y_\infty = \sum_{k=1}^\infty \frac{1}{n^s}$. Define a martingale measure $\P$ and its marginals $\mu, \nu$ as in Example \ref{ex:toolittleconvexity}.  Note that $\mu, \nu$ are compactly supported since $s>1$. Again we will show that this (optimal) $\P$ does not allow a dual maximizer, that is, there does not exist a triple $(f,g,h)$ which satisfies \eqref{dualineq3}, \eqref{dualeq3+}, \eqref{dualeq3-}, \eqref{dualeq30}, where $g$ is given as in \eqref{newg}. Again we will show that $g(y_\infty) = -\infty$, which is a contradiction to \eqref{dualeq3-}. To see this, again define $v_x$ as in \eqref{v_x} so that we have \eqref{v_xproperty1}, \eqref{v_xproperty2}. Next, let us consider the slope
\begin{align}
b_n = \frac{v_{x_n} (y_n) - v_{x_n} (y_{n-1})}{y_n - y_{n-1}}.
\end{align}
In order to estimate $b_n$, we will first estimate $b_n - b_{n+1}$. For this, as we can modify the $(f,g,h)$ by an affine function, we can assume that $f(x_{n+1}) = h(x_{n+1})=0$, thus without loss of generality we can assume that $b_{n+1} = 0$. Now notice that, because of  \eqref{dualineq3}, \eqref{dualeq3+}, \eqref{dualeq3-}, we must have the following inequality:

%Recall that $y_n - y_{n-1} = n^{-s}$. Then we compute
\begin{align*}
b_n - b_{n+1} &\ge \frac{v_{x_{n+1}} (y_n) - v_{x_{n+1}} (y_{n-1})}{y_n - y_{n-1}}\\
&=n^s[-|y_n - x_{n+1}|^r + |y_{n-1} - x_{n+1} |^r]\\
&=n^s\bigg{[}-\bigg{(}\frac{1}{2(n+1)^s}\bigg{)}^r + \bigg{(}   \frac{1}{2(n+1)^s} + \frac{1}{n^s}       \bigg{)}^r \bigg{]}\\
&=n^s\bigg{[}\frac{(2(n+1)^s +n^s)^r - n^{sr}}{2n^{sr}(n+1)^{sr}} \bigg{]}\\
& \approx Cn^s \cdot n^{sr} \cdot n^{-2sr} = Cn^{s-sr}. \quad\text{Hence we deduce that}         
\end{align*}
\begin{align*}
-b_n = \sum_{k=0}^{n-1} (b_k - b_{k+1})\gtrapprox Cn^{1+s-sr}. \quad \text{This implies that, since $sr <2$,}
 \end{align*}
\begin{align*}
(y_n - y_{n-1})b_n \lessapprox -Cn^{1-sr} \quad \Longrightarrow \quad \sum_{n=1}^\infty (y_n - y_{n-1})b_n = -\infty.
\end{align*}
Again as in Example \ref{ex:toolittleconvexity}, this tells us that $g(y_\infty) = -\infty$, a contradiction to \eqref{dualeq30}.
\end{example}
\begin{example}\label{nonintegrability}
In this example we show the necessity of semiconvexity for the regularity in Theorem \ref{regularity} and the Lipschitzness of $c$  alone is not sufficient, by constructing a $1$-Lipschitz cost $c$ and a compactly supported, irreducible pair $(\mu,\nu)$ for which $(f,g,h)$ is a dual maximizer, but $g \notin L^1(\nu)$. 

To do this, we take $c(x,y) = -|x-y|$, and let $I=]0,1[$ and  $\mu= {\rm Leb}\big{|}_{[0,1]}$. Choose a smooth and strictly concave function $\xi : I \to \R_-$ such that $\xi \le 0$, $\xi(\frac{1}{2})=0$, $\lim_{x \to 0^+} \xi(x) = \lim_{x \to 1^-} \xi(x) = -\infty$, and $\int_0^1 \xi(x) \mu(dx) = -\infty$. Now we will construct a probability measure $\nu$ where $\nu(I)=1$ and $(\mu,\nu)$ are irreducible, and also find a dual maximizer $(f,g,h)$ where $g = \xi$. Then $\int g(x) \nu(dx) \le \int g(x) \mu(dx) = -\infty$, as claimed.

To construct such $\nu$ and $(f,g,h)$,  observe that for each $x \in I$ there exist unique $f(x), h(x)$ such that the function $v_x(y) :=f(x)+h(x)\cdot(y-x) -|x-y|$ satisfies 
\begin{enumerate}
\item $v_x(y) \ge \xi(y) \quad \forall x \in I, \forall y \in I$, and
\item for each $x \in I$, $v_x$ is tangent to $\xi$ at two points, say $y^-(x), y^+(x)$.
\end{enumerate}
Note that then $y^-, y^+$ are well-defined on $I$, and $0<y^-(x) < x < y^+(x)<1$. Define a probability measure $\P_x := \frac{y^+(x)-x}{y^+(x)-y^-(x)}\delta_{y^-(x)} + \frac{x-y^-(x)}{y^+(x)-y^-(x)}\delta_{y^+(x)}$, and $\P \in P(\R^2)$ by $\P(dx,dy) = \P_x(dy) \cdot \mu(dx)$, i.e. $(\P_x)_x$ is a disintegration of $\P$ with respect to $\mu$. Define $\nu$ as the second marginal of $\P$ and note that by definition of $\P$, $(\mu, \nu)$ are irreducible and are concentrated on $I$. Now observe that the definition of $f,h$ gives us that $g(y) := \inf_{x \in I} \{v_x(y)\}$ satisfies $g = \xi$ so that $\int g(x) \nu(dx)= -\infty$, and $(f,g,h)$ is a dual maximizer with respect to $\mu, \nu$ and $c$.
\end{example}

\bibliography{joint_biblio}{}

\begin{thebibliography}{10}

\bibitem{BeCoHu14}
M.~{Beiglb{\"o}ck}, A.~Cox, and M.~Huesmann.
\newblock Optimal transport and {S}korokhod embedding.
\newblock {\em Invent. Math.}, 208(2):327--400, 2017.

\bibitem{BeHePe12}
M.~Beiglb{\"o}ck, P.~{Henry-Labord{\`e}re}, and F.~Penkner.
\newblock Model-independent bounds for option prices: A mass transport
  approach.
\newblock {\em Finance Stoch.}, 17(3):477--501, 2013.

\bibitem{BeJu16}
M.~Beiglb{\"o}ck and N.~Juillet.
\newblock On a problem of optimal transport under marginal martingale
  constraints.
\newblock {\em Ann. Probab.}, 44(1):42--106, 2016.

\bibitem{BeNu14}
M.~Beiglb{\"o}ck and M.~Nutz.
\newblock Martingale inequalities and deterministic counterparts.
\newblock {\em Electron. J. Probab.}, 2014.

\bibitem{BeNuTo16}
M.~{Beiglb{\"o}ck}, M.~{Nutz}, and N.~{Touzi}.
\newblock {Complete Duality for Martingale Optimal Transport on the Line}.
\newblock {\em Ann. Probab., to appear}, 2016.

\bibitem{CaLaMa14}
L.~{Campi}, I.~{Laachir}, and C.~{Martini}.
\newblock {Change of numeraire in the two-marginals martingale transport
  problem}.
\newblock {\em Finance Stoch.}, 21(2):471--486, June 2017.

\bibitem{DoSo12}
Y.~Dolinsky and H.~M. Soner.
\newblock Martingale optimal transport and robust hedging in continuous time.
\newblock {\em Probab. Theory Relat. Fields}, 160(1-2):391--427, 2014.

\bibitem{GaHeTo13}
A.~Galichon, P.~{Henry-Labord{\`e}re}, and N.~Touzi.
\newblock A stochastic control approach to no-arbitrage bounds given marginals,
  with an application to lookback options.
\newblock {\em Ann. Appl. Probab.}, 24(1):312--336, 2014.

\bibitem{GhKiLi16b}
N.~{Ghoussoub}, Y.-H. {Kim}, and T.~{Lim}.
\newblock {Structure of optimal martingale transport plans in general
  dimensions}.
\newblock arXiv:1508.01806, 2016.

\bibitem{HeObSpTo12}
P.~Henry-Labord\`ere, J.~Ob\l\'oj, P.~Spoida, and N.~Touzi.
\newblock The maximum maximum of a martingale with given {$n$} marginals.
\newblock {\em Ann. Appl. Probab.}, 26(1):1--44, 2016.

\bibitem{Ho11}
D.~Hobson.
\newblock The {S}korokhod embedding problem and model-independent bounds for
  option prices.
\newblock In {\em Paris-{P}rinceton {L}ectures on {M}athematical {F}inance
  2010}, volume 2003 of {\em Lecture Notes in Math.}, pages 267--318. Springer,
  Berlin, 2011.

\bibitem{HoKl13}
D.~{Hobson} and M.~{Klimmek}.
\newblock {Robust price bounds for the forward starting straddle}.
\newblock {\em Finance Stoch.}, 9(1):189--214, April 2015.

\bibitem{HoNe12}
D.~Hobson and A.~Neuberger.
\newblock Robust bounds for forward start options.
\newblock {\em Math. Finance}, 22(1):31--56, 2012.

\bibitem{NuStTa17}
M.~{Nutz}, F.~{Stebegg}, and X.~{Tan}.
\newblock {Multiperiod Martingale Transport}.
\newblock {\em ArXiv e-prints}, March 2017.

\bibitem{Ob04}
J.~Ob{\l}{\'o}j.
\newblock The {S}korokhod embedding problem and its offspring.
\newblock {\em Probab. Surv.}, 1:321--390, 2004.

\bibitem{ObSp14}
J.~Ob{\l}{\'o}j, P.~Spoida, and N.~Touzi.
\newblock Martingale inequalities for the maximum via pathwise arguments.
\newblock In {\em In {M}emoriam {M}arc {Y}or-S{\'e}minaire de
  {P}robabilit{\'e}s XLVII}, pages 227--247. Springer, 2015.

\bibitem{St65}
V.~Strassen.
\newblock The existence of probability measures with given marginals.
\newblock {\em Ann. Math. Statist.}, 36:423--439, 1965.

\end{thebibliography}
\bibliographystyle{plain}
\end{document}